\documentclass[11pt]{amsart}
\usepackage{amssymb}

\newtheorem{thm}{Theorem}
\newtheorem{prop}[thm]{Proposition}
\newtheorem{lem}[thm]{Lemma}

\theoremstyle{remark}

\theoremstyle{definition}

\newcommand{\col}{\kern -3pt :}
\newcommand{\C}{\mathbb C}
\newcommand{\R}{\mathbb R}

\newcommand{\N}{\mathbb N}

\newcommand{\HH}{\mathbb H}
\newcommand{\SSS}{\mathbb S}

\newcommand{\T}{\mathcal T}

\newcommand{\E}{\mathrm{e}}
\newcommand{\I}{\mathrm{i}}

\renewcommand{\leq}{\leqslant}
\renewcommand{\geq}{\geqslant}
\renewcommand{\phi}{\varphi}

\title[Infinitesimal Liouville currents]{Infinitesimal Liouville currents, cross-ratios and intersection numbers}

\author{Francis Bonahon}
\address {F.B.: Department
of Mathematics,  University of
Southern California, Los Angeles
CA~90089-2532, U.S.A.}
\email{fbonahon@math.usc.edu}

\author{Dragomir \v Sari\'c}
\address{D.S.: Department
of Mathematics, Queens College, City University of New York, Flushing NY~11367, U.S.A.}
\address{D.S.: Mathematics Ph.D. Program, 
The CUNY Graduate Center, 
365 Fifth Avenue, 
Room 4208, 
New York NY 10016-4309, U.S.A.}
\email{dragomir.saric@qc.cuny.edu}

\date{\today}

\thanks{This research was partially supported by the grant  DMS-0604866 from the National Science Foundation.}

\begin{document}
\maketitle

\begin{abstract}
Many classical objects on a surface $S$ can be interpreted as cross-ratio functions on the circle at infinity of the universal covering $\widetilde S$. This includes closed curves considered up to homotopy, metrics of negative curvature considered up to isotopy and, in the case of interest here, tangent vectors to the Teichm\"uller space of complex structures on  $S$. When two cross-ratio functions are sufficiently regular, they have a geometric intersection number, which generalizes the intersection number of two closed curves. In the case of the cross-ratio functions associated to tangent vectors to the Teichm\"uller space, we show that two such cross-ratio functions have a well-defined geometric intersection number, and that this intersection number is equal to the Weil-Petersson scalar product of the corresponding vectors.  
\end{abstract}

Let $S$ be a compact orientable surface of negative Euler characteristic. Its universal cover $\widetilde S$ has a well-defined \emph{circle at infinity} $\partial_\infty \widetilde S$, which can be described uniquely in terms of the topology of $S$, or even in terms of the fundamental group $\pi_1(S)$ \cite{Gromov, GhysHarpe, CDP, BridgsonHaefliger}. The action of $\pi_1(S)$ on the universal cover continuously extends to $\widetilde S \cup \partial_\infty \widetilde S$. 

A \emph{cross-ratio function} (or \emph{cross-ratio} for short) is a function $\alpha$ that associates a real number $\alpha(I_1, I_2) \in \R$ to each pair of two intervals $I$, $J\subset \partial_\infty S$ with disjoint closures, and that satisfies the following two conditions:
\begin{enumerate}
\item (Finite Additivity) $\alpha (I, J) = \alpha (I_1, J) + \alpha (I_2, J)$ whenever the interval $I$ is split as the union of two disjoint intervals $I_1$ and $I_2$;
similarly, $\alpha (I, J) = \alpha (I, J_1) + \alpha (I,J_2)$ whenever the interval $J$ is split as the union of two disjoint intervals $J_1$ and $J_2$.
\item (Invariance) $\alpha $ is invariant under the action of the fundamental group, in the sense that $\alpha \bigl( \gamma(I), \gamma(J) \bigr)=\alpha (I, J)$ for every $\gamma\in \pi_1(S)$ and every intervals $I$, $J\subset \partial_\infty \widetilde S$. 
\end{enumerate}
Here an interval is allowed to be closed, semi-open of open according to whether it includes all, some, or none of its end points, respectively. We let $\mathcal X(S)$ denote the space of all cross-ratio functions. 

Most of the cross-ratio functions that we will consider in this article will be, in addition, \emph{symmetric} in the sense that $\alpha (I, J) = \alpha(J, I)$ for every $I$ and $J$. However, this property is not crucial.

If, in addition, a cross-ratio $\alpha \in \mathcal X(S)$ takes only non-negative values, then it is countably additive and consequently defines a $\pi_1(S)$--invariant Radon measure on the space $\partial_\infty\widetilde S \times \partial_\infty\widetilde S - \Delta$, where $\Delta$ denotes the diagonal of the product. Such a measure is a \emph{measure geodesic current}. 

A fundamental example of cross-ratio function is associated to a complex structure $m$ on the surface $S$. Then, the universal covering $\widetilde S$ is biholomorphically equivalent to the open unit disk $\HH^2 \subset \C$, which provides an identification of $\partial_\infty \widetilde S$ with the circle $\SSS^1$ bounding this disk, well-defined up to a linear fractional map preserving $\HH^2$. We can then consider the \emph{Liouville geodesic current} $L_m$, defined by the property that
$$
L_m (I, J) = \left|  \log \frac{(a-c)(b-d)}{(a - d)(b-c)}  \right|
$$
if $a$, $b$ and $c$, $d$ are the end points of $I$ and $J$, respectively. This example explains the terminology, and the connection with the classical cross-ratios. See \cite[Theorem~13]{Bonahon2} for a characterization of which cross-ratios occur in this way. See also \cite{Bonahon1, Otal1, Otal3, Bonahon3, Labourie1, Labourie2, LabourieMcShane} for various incarnations of cross-ratio functions.

This article is devoted to infinitesimal versions of these Liouville cross-ratios. Let $\T(S)$ be the \emph{Teichm\"uller space} of $S$, considered as the space of isotopy classes of complex structures on $S$, and let $V\in T_{m_0} \T(S)$ be a vector tangent to $\T(S)$ at $m_0$. If $t\mapsto m_t$ is a curve in $\T(S)$ passing through $m_0$ and tangent to $V$ at $t=0$, we can consider the derivative
$$
L_V(I, J) = \frac\partial{\partial t} L_{m_t} (I, J) _{|t=0}
$$
for every pair  $I$, $J\subset \partial_\infty \widetilde S$ of intervals with disjoint closures. It is fairly well-known that the derivative exists, and depends only on $V\in T_{m_0} \T(S)$, and not on the curve $t\mapsto m_t$ tangent to $V$.  This $L_V$ is the \emph{infinitesimal Liouville cross-ratio} associated to $V\in T_{m_0}\T(S)$. 

Such an infinitesimal Liouville cross-ratio does not induce a measure on $\partial_\infty\widetilde S \times \partial_\infty\widetilde S - \Delta$ any more. It has a weaker regularity property, in the sense that it only defines a \emph{H\"older geodesic current}, namely a $\pi_1(S)$--invariant linear functional on the space of H\"older continuous functions with compact support on $\partial_\infty\widetilde S \times \partial_\infty\widetilde S - \Delta$. See \cite{BonahonSozen, Saric1, Otal2}. 

The space $\mathcal C(S)$ of measure geodesic currents on $S$ was introduced in \cite{Bonahon1}  to construct a completion of the space of homotopy classes of weighted closed curves in $S$. A fundamental feature of this space is a continuous function
$$
i \col \mathcal C(S) \times \mathcal C(S) \to \R,
$$
which extends the geometric intersection function for closed curves. The proof that this geometric intersection function $i$ is finite and continuous heavily depends on the regularity of measure geodesic currents.

For general cross-ratio functions $\alpha$, $\beta \in \mathcal X(S)$, it is formally possible to copy the above construction and attempt to define a \emph{geometric intersection number} $i(\alpha, \beta)$.  However, making sense of this intersection number amounts to proving the convergence of a certain infinite sum, and requires additional regularity hypotheses on the cross-ratios. 

The main contribution of this article is the rigorous construction of intersection numbers for a class of cross-ratios which includes infinitesimal Liouville cross-ratios.

Given $\nu>0$, a cross-ratio function $\alpha \in \mathcal X(S)$ is said to be \emph{$\nu$--H\"older regular} with respect to a complex structure $m_0 \in \T(S)$ if there exist constants $c_0$, $c_1>0$ such that
$$
\left| \alpha(I,J) \right| \leq c_0 L_{m_0}(I,J)^\nu
$$
for all intervals $I$, $J\subset \partial_\infty \widetilde S$ with disjoint closure and such that $ L_{m_0}(I, J) \leq c_1$; recall that $L_{m_0}$ denotes the Liouville geodesic current associated to the complex structure $m_0$. A H\"older regular cross-ratio function defines a H\"older geodesic current (compare \cite{Saric1}), but the converse is not true; for instance, a measure geodesic current with an atom (such as the one associated to a homotopy class of closed curves) is a H\"older geodesic current, but is not H\"older regular in the above sense. 

\begin{thm}
\label{thm:Main1}
If the cross-ratios $\alpha$, $\beta \in \mathcal X(S)$ are $\nu$--H\"older regular with respect to some complex structure $m_0 \in \T(S)$ and for some $\nu > \frac 34$, then it is possible to define a geometric intersection number $i(\alpha, \beta)$, in a sense made precise by Theorem~{\upshape \ref{thm:Main3}} below. 
\end{thm}

The construction is based on a certain bundle over $S$ with fiber the open annulus, it uses a covering of this bundle by ``double boxes'', and it heavily relies on  a relatively subtle growth estimate on the sizes of these double boxes. See \S \ref{sect:BoxCount}--\ref{sect:InterNumbHolder}. 

A much easier property is that Theorem~\ref{thm:Main1} can be applied to infinitesimal Liouville currents: 

\begin{prop}
\label{prop:InfLiouvRegular}
If $V\in T_{m_0}\T(S)$ is a vector tangent to the Teichm\"uller space at $m_0 \in \T(S)$, then the associated infinitesimal Liouville cross-ratio $L_V\in \mathcal X(S)$  is $\nu$--H\"older regular with respect to $m_0$ for every $\nu<1$. 
\end{prop}

As a consequence, given two such tangent vectors $V$, $W\in T_{m_0} \T(S)$, we can make sense of the  geometric intersection number $i(L_V, L_W)$ of their infinitesimal Liouville cross-ratios. 

\begin{thm}
\label{thm:Main2}
Let $t\mapsto m_t$ and $u \mapsto n_u$ be two differentiable curves in $\T(S)$, respectively tangent to the vectors $V$ and $W\in T_{m_o}\T(S)$ at $m_0=n_0\in \T(S)$. Then,
$$
\frac\partial{\partial t} \frac\partial{\partial u} i(L_{m_t}, L_{n_u})_{|(t,u)=(0,0)} = i(L_V, L_W),
$$
where $i(L_V, L_W)$ is the geometric intersection number provided by Theorem~{\upshape\ref{thm:Main1}} and Proposition~{\upshape\ref{prop:InfLiouvRegular}}, and where $i(L_{m_t}, L_{n_u})$ is the classical intersection number of measure geodesic currents as in \cite{Bonahon1}. 
\end{thm}

If we combine Theorem~\ref{thm:Main2} with earlier work of Thurston and Wolpert \cite{Wolpert2}, we automatically obtain:

\begin{thm}
\label{thm:WP}
Under the hypotheses of Theorem~{\upshape\ref{thm:Main1}}, the geometric intersection number $i(L_V, L_W)$ is equal, up to multiplication by a constant,  to the scalar product $\omega_{\mathrm{WP}}(V, W)$ of the tangent vectors $V$, $W\in T_{m_0} \T(S)$ under the Weil-Petersson metric of $\T(S)$. \qed
\end{thm}

The constant depends on the topological type of the surface and on the conventions in the definition of the Weil-Petersson metric; see \cite{Wolpert2}.

To a large extent, Theorems~\ref{thm:Main1} and \ref{thm:Main2} complete the analogy between geometric intersection numbers and Weil-Petersson metric that was proposed in \cite{Bonahon2} (see also \cite{Wolpert3}). In particular, they provide a much more satisfactory framework than the clumsy construction of \cite[\S4]{Bonahon2}, which had been designed to bypass the analytic subtleties caused by the lack of regularity of infinitesimal Liouville cross-ratios.

\section{Geometric intersection numbers}
\label{sect:InterNumber}

This section is mostly heuristic, and summarizes the definition of geometric intersection numbers in the case of measure geodesic currents. 

It is conceptually convenient to endow the surface $S$ with a complex structure $m_0$. Then, the space $\partial_\infty\widetilde S \times \partial_\infty\widetilde S - \Delta$ has a natural identification with the space $G(\widetilde S)$ of oriented complete geodesics for the Poincar\'e metric of $\widetilde S$, since such a geodesic joins two distinct points of  the circle at infinity $\partial_\infty \widetilde S$. 

Let $DG(\widetilde S)$ be the \emph{double geodesic space}, consisting of all pairs $(g,h)$ of geodesics $g$, $h \in G(\widetilde S)$ which transversely meet at some point. Considering this intersection point $g\cap h$ and the tangent vectors of $g$ and $h$ at this point, we see that $DG(\widetilde S)$ can also be identified to the set of triples $(\widetilde x, v, w)$ consisting of a point  $\widetilde x\in \widetilde S$ and of two distinct unit tangent vectors $v$, $w\in T_{\widetilde x} \widetilde S$ at $\widetilde x$. This description makes it clear that the action of $\pi_1(S)$ on $DG(\widetilde S)$ is free and discontinuous, so that we can consider the quotient $DG(S) = DG(\widetilde S)/\pi_1(S)$.

Again, $DG(S)$ can be interpreted as the set of triples $(x, v, w)$ consisting of a point  $x\in S$ and of two distinct unit tangent vectors $v$, $w\in T_x S$ at $x$. In particular, it is a manifold of dimension 4. Note that it is non-compact, which is the major cause of the analytic problems that we will encounter. 

Let a \emph{box} in $G(\widetilde S) \subset \partial_\infty\widetilde S \times\partial_\infty\widetilde S  $ be a subset of the form $I\times J$, where $I$ and $J$ are intervals with disjoint closures in $\partial_\infty\widetilde S$. A \emph{double box} in $DG(\widetilde S) \subset G(\widetilde S) \times G(\widetilde S)$ is a subset of the form $B= B_1 \times B_2$ where $B_1$ and $B_2$ are two boxes of $G(\widetilde S)$ such that every geodesic $g\in B_1$ crosses every geodesic $h\in B_2$.

Finally, let a \emph{double box} in the quotient $DG(S) = DG(\widetilde S)/\pi_1(S)$ be a subset $B$ which is the image of a double box $\widetilde B$ in $DG(\widetilde S)$ small enough that $DG(\widetilde S) \to DG(S)$ is injective on the closure of $\widetilde B$, and consequently restricts to a homeomorphism $\widetilde B \to B$. 

Although the consideration of geodesics is convenient and more intuitive, the reader will notice that the spaces $G(\widetilde S)$, $DG(\widetilde S)$ and $DG(S)$, as well as the notions of boxes and double boxes, can be described without any reference to  a complex structure $m_0$ on $S$.

For future reference, we note the following immediate property.
\begin{lem}
\label{lem:SplitBox}
If $B_1$ and $B_2$ are two double boxes in $DG(S)$, their intersection $B_1 \cap B_2$ is a double box, and the complement $B_1 - B_2$ can be decomposed as the union of finitely many disjoint double boxes. \qed
\end{lem}

\begin{lem}
\label{lem:BoxCover}
The space $DG(S)$ can be decomposed as the union of a locally finite family of disjoint double boxes $\{ B_i \}_{i\in I}$. 
\end{lem}
\begin{proof}
It should be clear from definitions that every element of $DG(S)$ is contained in the interior of some double box. We can therefore write $DG(S)$ as the union of a locally finite family of double boxes. We can then arrange that these double boxes are disjoint by successive applications of Lemma~\ref{lem:SplitBox}. 
\end{proof}

A cross-ratio function $\alpha \in \mathcal X(S)$ associates a number $\alpha(B) = \alpha(I_1, I_2)$  to each box $B = I_1 \times I_2$ in $G(\widetilde S)$. 

Two cross-ratio functions $\alpha$, $\beta \in \mathcal X(S)$ associate a number $\alpha\times\beta(B) = \alpha(B_1) \beta(B_2)$ to each double box $B = B_1 \times B_2$ in $DG(\widetilde S)$. Finally, if $B$ is a double box in $DG(S)$ image of a double box $\widetilde B \subset DG(\widetilde S)$, define $\alpha \times \beta(B) = \alpha \times \beta(\widetilde B)$. The invariance of $\alpha$ and $\beta$ under the action of $\pi_1(S)$ guarantees that $\alpha \times \beta(B)$ depends only on $\alpha$, $\beta$ and $B$, and not on the double box $\widetilde B$ lifting $B$ to $DG(\widetilde S)$. 

We would like to define the \emph{geometric intersection number} $\I(\alpha, \beta)$ of the cross-ratios $\alpha$, $\beta \in \mathcal X(S)$  as the infinite sum
$$
\I(\alpha, \beta) = \sum_{i\in I} \alpha\times \beta(B_i)
$$
for some decomposition $DG(S) = \bigcup_{i\in I} B_i$ as in Lemma~\ref{lem:BoxCover}. 

When $\alpha$ and $\beta$ are measure geodesic currents, this sum is proved to be (absolutely) convergent in \cite[\S 4.2]{Bonahon1}, and this for any decomposition of $DG(S)$ into disjoint double boxes. 

However, for cross-ratio functions $\alpha$, $\beta \in \mathcal X(S)$, we need to find a scheme that provides a decomposition $DG(S) = \bigcup_{i\in I} B_i$ for which the above sum converges, and is independent of the decomposition of $DG(S)$ into double boxes provided by that scheme.

We actually will not quite carry out this program, and our construction will only  use a locally finite decomposition into double boxes of a suitable \emph{open dense subset} of  $DG(S)$.

\section{Good coverings by double boxes}
\label{sect:BoxCount}

In many of the estimates of the article, we say that the quantity $X$ is \emph{of order at most} $Y$, and we write $X\prec Y$, if there exists a constant $c>0$ such that $X\leq cY$. We say that $X$ is \emph{of the same order as} $Y$, and we write $X\asymp Y$, if $X\prec Y$ and $Y\prec X$. 

Choose a complex structure $m_0\in \T(S)$ and a base point $\widetilde x_0\in \widetilde S$. This defines a riemannian metric on the circle at infinity $\partial_\infty \widetilde S$, where the distance between $\eta$, $\xi \in \partial_\infty\widetilde S$ is defined as the angle between the Poincar\'e geodesics joining $\widetilde x_0$ to $\eta$ and $\xi$, respectively. Taking a different base point $\widetilde x_0\in \widetilde S$ modifies this metric only up to bi-Lipschitz equivalence. However, if we change the complex structure  $m_0\in \T(S)$, the new metric is usually only bi-H\"older equivalent to the original one. We  state this property for future reference.

\begin{lem}
\label{lem:HolderStructure}
On the circle at infinity $\partial_\infty \widetilde S$, let $d_0$ be the metric induced as above by the choice of a complex structure $m_0 \in \T(S)$ and of a base point $\widetilde x_0 \in \widetilde S$, and let $d_1$ be similarly associated to $m_1 \in \T(S)$ and $\widetilde x_1 \in \widetilde S$. Then there exists $\nu\leq 1$ such that 
$$
d_0(\xi, \xi')^{\frac1\nu} \prec d_1(\xi, \xi') \prec d_0(\xi, \xi')^\nu  
$$
for all $\xi$, $\xi'\in \partial_\infty \widetilde S$.

In addition, $\nu$ tends to $1$ as $m_1$ tends to $m_0$ in $\T(S)$.   \qed
\end{lem}

This riemannian metric on the circle at infinity $\partial_\infty \widetilde S$ gives a riemannian metric on the geodesic space $G(\widetilde S)\subset \partial_\infty \widetilde S \times \partial_\infty \widetilde S$, and therefore on the product $ G(\widetilde S) \times G(\widetilde S)$. 

With this data, a relatively compact subset $X\subset G(\widetilde S) \times G(\widetilde S)$ has \emph{Minkowski $m_0$--dimension} $\leq d$ if  the volume of the $\varepsilon$--neighborhood $U_\varepsilon$ of $X$ in $G(\widetilde S) \times G(\widetilde S)$ is such that 
$$
\mathrm{vol} (U_\varepsilon) \prec \varepsilon^{4-d}
$$
as $\varepsilon>0$ is bounded above.  This definition is clearly independent of the choice of base point $\widetilde x_0\in \widetilde S$, but does depend on the complex structure $m_0\in \T(S)$. The Minkowski dimension is also often called the \emph{box counting dimension}, but this terminology would be here clumsy since we are already dealing with many types of boxes. Note that $d\leq 4$, and that a separating subset necessarily has Minkowski $m_0$--dimension $\geq 3$.

Recall that the double geodesic space $DG(\widetilde S)$ consists of all pairs of geodesics $(g,h) \in G(\widetilde S) \times G(\widetilde S)$ such that $g$ and $h$ transversely meet in one point.

A \emph{subdivision scheme} for a subset $\Omega\subset DG(\widetilde S)  $ consists of two families $\{ \mathcal I_n \}_{n \in \N}$ and  $\{ \mathcal B_n \}_{n \in \N}$ such that:
\begin{enumerate}
\item each $I\in \mathcal I_n$ is an interval in the circle at infinity $\partial_\infty \widetilde S$;
\item the family $\mathcal I_{n+1}$ is obtained from $\mathcal I_n$ by subdividing each interval $I\in \mathcal I_n$ into two intervals;
\item each $B\in \mathcal B_n$ is a double  box $I_1 \times I_2 \times I_3 \times I_4$ in $DG(\widetilde S) \subset (\partial_\infty \widetilde S)^4$ where $I_1$, $I_2$, $I_3$, $I_4$ are intervals of $\mathcal I_n$;
\item $\Omega = \bigcup_{n=1}^\infty \bigcup_{B \in \mathcal B_n} B$;
\item any two double boxes $B\in \mathcal B_n$ and $B'\in \mathcal B_{n'}$ are disjoint.
\end{enumerate}

The intervals $I\in \mathcal I_n$ can be closed, open or semi-open. In particular, in Condition (5), the {closures} of two boxes are allowed to have a non-trivial intersection. We actually will not worry much about the box boundaries, as they are irrelevant for the type of cross-ratio functions considered in the rest of the article.

\begin{lem}
\label{lem:GoodBoxes}
Let $\Omega$ be an open  subset of the double geodesic space $DG(\widetilde S) \subset  G(\widetilde S) \times G(\widetilde S)$ that is relatively compact  in $G(\widetilde S) \times G(\widetilde S)$ and   whose topological frontier $\delta\Omega $  in $G(\widetilde S) \times G(\widetilde S)$ has Minkowski $m_0$--dimension $\leq d$. Pick two numbers $0 < r\leq \frac12 \leq R<1$. Then  there exists a subdivision scheme $\{ \mathcal I_n \}_{n \in \N}$,  $\{ \mathcal B_n \}_{n \in \N}$, such that:

\begin{enumerate}
\item the length of each interval  $I \in \mathcal B_n$ is of order between $r^n$ and $R^n$, for the metric on $\partial_\infty \widetilde S$ defined by the complex structure $m_0$ and by a choice of base point in $\widetilde S_0$;
\item each double box $B \in \mathcal B_n$ is at distance $\prec R^n$ from the complement of $\Omega$ in $G(\widetilde S) \times G(\widetilde S)$;
\item the number of boxes in $\mathcal B_n$ is $\prec r^{-dn}$. 
\end{enumerate}
\end{lem}

\begin{proof} We will prove the result in the case where $r=R=\frac12$, which of course implies the general case. (The result is stated in the above form for future reference).

Since $\Omega$ is relatively compact  in $G(\widetilde S) \times G(\widetilde S)$, there exists a finite family of boxes $A_1$, $A_2$, \dots, $A_s$ in $G(\widetilde S)$ such that $\Omega$ is contained in the union of the products $A_i \times A_j$. 

Each box $A_i$ is of the form $A_i = I_i \times J_i$, where $I_i$ and $J_i$ are two disjoint intervals in the circle at infinity $\partial_\infty \widetilde S$. Let $\mathcal I_1$ be the family of these finitely many intervals. By subdivision of the boxes $A_i$, we can arrange that these intervals are disjoint. 

We now define the sequence $\mathcal I_n$ by induction, where $\mathcal I_1$ is the family of the above intervals $I_i$, $J_i$,  and where $\mathcal I_{n+1}$ is obtained from $\mathcal I_n$ by subdividing each interval into two intervals of equal lengths. In particular, the length of each $\I\in \mathcal I_n$ is of order $2^{-n}$.

Define the family $\mathcal B_n$ to consist of all double boxes $B =I_1 \times I_2 \times I_3 \times I_4$ with all four $I_i$ in $\mathcal I_n$, such that:
\begin{enumerate}
\item the double box $B$ is contained in $\Omega$;
\item if $I_i' \in \mathcal I_{n-1}$ is the level $n-1$ interval that contains $I_i$, the double box $B' =I_1' \times I_2' \times I_3' \times I_4'$ is not  contained in $\Omega$ (so that $B' \not\in \mathcal B_{n-1}$). 
\end{enumerate}

It is immediate from the construction that any two double boxes $B\in \mathcal B_n$ and $B'\in \mathcal B_{n'}$ are disjoint.

Also, the union $\bigcup_{n=1}^\infty \bigcup_{B \in \mathcal B_n} B$  is equal to $\Omega$. Indeed, for every $n$, each pair of geodesics $(g,f) \in \Omega$ is contained in some double box $B =I_1 \times I_2 \times I_3 \times I_4$ with all four $I_i$ in $\mathcal I_n$. This double box $B$ will  be contained in $\Omega$ for $n$ large enough since $\Omega$ is open;  this box will belong to $\mathcal B_n$ for the first such $n$. 

Therefore,  $\{ \mathcal I_n \}_{n \in \N}$ and  $\{ \mathcal B_n \}_{n \in \N}$  provide a subdivision scheme for $\Omega$, and the lengths of the intervals of $\mathcal I_n$ are of order $2^{-n}$. 

The construction also makes it clear that every $\widetilde B \in \mathcal B_n$ is at distance $\prec 2^{-n}$ from the complement of $\Omega$ in $G(\widetilde S) \times G(\widetilde S)$, and therefore at distance $\prec 2^{-n}$ from the frontier $\delta\Omega$ of $\Omega$. Indeed, the level $n-1$ double box that contains $\widetilde B$ is not contained in $\Omega$. 

In particular, each $B \in \mathcal B_n$ is contained in the $\epsilon_n$--neighborhood of the frontier $\delta\Omega$, where $\epsilon_n \asymp2^{-n}$. By definition of Minkowski $m_0$--dimension, the volume of this neighborhood is $\prec 2^{-(4-d)n}$. On the other hand, the volume of each $B \in \mathcal B_n$ is of order $2^{-4n}$. Since the double boxes have disjoint interior, we conclude that the number of $B \in \mathcal B_n$ is of order at most~$2^{dn}$.

This concludes the proof of Lemma~\ref{lem:GoodBoxes} when $r=R=\frac12$, and therefore in the general case. 
\end{proof}

\section{Intersection number of H\"older regular cross-ratios}
\label{sect:InterNumbHolder}

Now, consider two  cross-ratio functions  $\alpha$, $\beta \in \mathcal X(S)$  that are $\nu$--H\"older regular with respect to $m_0$, as defined in the introduction.

For an open subset  $\Omega \subset DG(\widetilde S)$ satisfying the hypotheses of Lemma~\ref{lem:GoodBoxes}, let  $\mathcal B = \bigcup_{n\in \N} \mathcal B_n$ be the family of double boxes covering $\Omega$ provided by that statement. 
As in \S \ref{sect:InterNumber}, each double box $ B\in {\mathcal B}$ is the product $ B =  B_1 \times B_2$ of two boxes in the geodesic space $G(\widetilde S)$, and we can define $\alpha\times \beta (B) = \alpha(B_1) \beta(B_2)$ and
$$
\mathrm i_{\Omega} (\alpha, \beta) = \sum_{B \in \mathcal B} \alpha \times \beta(B).
$$

\begin{lem}
\label{lem:IntersectionConverges}
Under the hypotheses of Lemma~{\upshape\ref{lem:GoodBoxes}}, let $\alpha $, $\beta \in \mathcal X(S)$  be $\nu$--H\"older regular with respect to $m_0$. If   $\nu > \frac d4  \frac{\log r}{\log R}$ and  if $\mathcal B = \bigcup_{n\in \N} \mathcal B_n$ is the family of double boxes provided by Lemma~{\upshape\ref{lem:GoodBoxes}}, the sum
$$
\mathrm i_{\Omega} (\alpha, \beta) = \sum_{B \in \mathcal B} \alpha \times \beta(B).
$$
is (absolutely) convergent.
\end{lem}

\begin{proof} For a box $B  \subset G(\widetilde S)$ containing at least one geodesic of $\Omega$ (so that it is at bounded distance from the base point of $\widetilde S$), the Liouville mass $L_{m_0}(B)$ is of the same order as the product of its side lengths. Therefore, if the double box  $B = B_1 \times B_2$ is in $\mathcal B_n$, 
$$
\alpha\times \beta(B) = \alpha(B_1) \beta(B_2)
\prec L_{m_0}(B_1)^\nu L_{m_0}(B_2)^\nu 
\prec R^{4\nu n}
$$
because $B_1$ and $B_2$ have side lengths of order $\prec R^n$ and meet a fixed compact subset of $G(\widetilde S)$, namely the union of the images of $\Omega$ under the two projections $G(\widetilde S) \times G(\widetilde S) \to G(\widetilde S)$. 

As a consequence, since the number of elements of $\mathcal B_n$ is of order at most $r^{-dn}$, 
$$
 \sum_{B \in \mathcal B} \alpha \times \beta(B)
 =  \sum_{n=1}^\infty \sum_{B \in \mathcal B_n} \alpha \times \beta(B)
 \prec \sum_{n=1}^\infty r^{-dn} R^{4\nu n}
 <\infty
$$
if $\nu > \frac d4  \frac{\log r}{\log R}$.
\end{proof}

\begin{lem}
\label{lem:IntIndeptBoxes}
Under the hypotheses of Lemma~{\upshape\ref{lem:IntersectionConverges}}, suppose in addition that  $\nu > 2\left( \frac{\log r}{\log R} -1  \right) + \frac d4$. Then the sum
$$
\mathrm i_{\Omega} (\alpha, \beta) = \sum_{B \in \mathcal B} \ \alpha \times \beta(B).
$$
is independent of the subdivision scheme providing the family of double boxes~$\mathcal B$.
\end{lem}

\begin{proof}
Let $\mathcal B = \bigcup_{n=1}^\infty \mathcal B_n $ and $\mathcal B' = \bigcup_{n=1}^\infty \mathcal B_n' $ be two families of double boxes as in Lemma~\ref{lem:GoodBoxes}, respectively associated to families $\{\mathcal I_n \}_{n \in \N}$ and $\{\mathcal I_n' \}_{n \in \N}$ of intervals in $\partial_\infty \widetilde S$. 

For $k\leq n$, the subdivision scheme enables us to decompose each double box $B\in \mathcal B_k$ into $16^{n-k}$ double boxes $B' = I_1 \times I_2 \times I_3 \times I_4$ with all $I_i \in \mathcal I_n$.  In particular, the side lengths of these new boxes are of order between $r^n$ and $R^n$. Let $\mathcal B_{k,n}$ be the family of double boxes so obtained. Similarly define a family $\mathcal B_{k,n}'$ by subdividing each double box $B'\in \mathcal B_k'$ into $16^{n-k}$ double boxes whose side lengths are of order between $r^n$ and $R^n$. Consider the families  ${\mathcal C}_{n} = \bigcup_{k=1}^n \mathcal B_{k,n}$ and ${\mathcal C}_{n}' = \bigcup_{k=1}^n \mathcal B_{k,n}'$ of all double boxes so created.

Because their lengths are of order between $r^n$ and $R^n$, each interval in $\mathcal I_n$ meets at most $\prec \frac{R^n}{r^n}$ intervals of $\mathcal I_n'$, and conversely each interval in $\mathcal I_n'$ meets $\prec \frac{R^n}{r^n}$ intervals of $\mathcal I_n$. We can therefore subdivide the  double boxes of ${\mathcal C}_{n}$ and ${\mathcal C}_{n}'$ into a common family  ${\mathcal D}_{n}$ of disjoint double boxes,  in such a way that each double box of ${\mathcal C}_{n}$ and ${\mathcal C}_{n}'$ is the union of a  number $\prec \frac{R^{4n}}{r^{4n}}$ of double boxes of~${\mathcal D}_{n}$.

We split the family ${\mathcal D}_{n}$ into three disjoint families $\mathcal D_{n}^{(0)}$, $\mathcal D_{n}^{(1)}$ and  $\mathcal D_{n}^{(2)}$, where $\mathcal D_{n}^{(0)}$ consists of those $  D \in \mathcal D_{n}$ that are contained in both $\bigcup_{ C \in \mathcal C_{n}} C$ and $\bigcup_{C' \in \mathcal C_{n}'}  C'$,  where $\mathcal D_{n}^{(1)}$ consists of those $ D$ that are contained in $\bigcup_{ C \in \mathcal C_{n}}  C$ but not in  $\bigcup_{ C' \in \mathcal C_{n}'}  C'$, and where $\mathcal D_{n}^{(2)}$ consists of those $  D $ that are contained in $\bigcup_{ C' \in \mathcal C_{n}'}  C'$ but not in  $\bigcup_{ C\in \mathcal C_{n}}  C$. 

We now use Condition~(3) of Lemma~\ref{lem:GoodBoxes}, to show that the number of elements of $\mathcal D_{n}^{(1)}$ is smaller than one could have expected. Indeed, by definition, such a double box $  D \in \mathcal D_{n}^{(1)}$ is contained in some double box $C \in \mathcal C_n$ and not in $\bigcup_{ C'\in \mathcal C_{n}'}  C' = \bigcup_{k=1}^n \bigcup_{ B'\in \mathcal B_k'}  B'$. It therefore meets some double box $ B'\in \mathcal B_{m}'$ with $m>n$. By Condition~(3) of Lemma~\ref{lem:GoodBoxes}, the double box $C$ consequently is at distance $\prec R^n$ from the complement of $\Omega$, and is therefore contained in an $\epsilon_n$--neighborhood of the boundary $\delta\Omega$ with $\epsilon_n \asymp R^n$ since its diameter is $\prec R^n$. By a volume estimate, it follows that there can be at most $\prec \frac{R^{(4-d)n}}{r^{4n}}$ such double boxes $C \in \mathcal C_n$ containing a double box $D \in \mathcal D_{n}^{(1)}$. Since each  $C \in \mathcal C_n$ contains at most $\prec \frac{R^{4n}}{r^{4n}}$ double boxes of ${\mathcal D}_{n}$, it follows that the number of elements of $\mathcal D_{n}^{(1)}$ is of order at most $\frac{R^{4n}}{r^{4n}} \frac{R^{(4-d)n}}{r^{4n}} = \frac{R^{(8-d)n}}{r^{8n}}$

The same argument shows that $ \mathcal D_{n}^{(2)}$ has $ \prec\frac{R^{(8-d)n}}{r^{8n}}$ elements. 

We are now ready to complete the proof of Lemma~\ref{lem:IntIndeptBoxes}. To show that the two families of double boxes $\mathcal B$ and $\mathcal B'$ give the same value for $i_\Omega(\alpha,\beta)$, we need to prove that $\sum_{B \in \mathcal B} \alpha \times \beta (B) = \sum_{ B' \in \mathcal B'} \alpha \times \beta( B')$, and therefore that
$$
\sum_{n=1}^\infty \sum_{ B \in \mathcal B_n} \alpha \times \beta( B) 
=\sum_{n=1}^\infty \sum_{ B' \in \mathcal B_n'} \alpha \times \beta( B').
$$ 
Consider a partial sum of the first series. By finite additivity of $\alpha $ and $\beta$, 
\begin{align*}
\sum_{k=1}^n \sum_{ B \in \mathcal B_k} \alpha \times \beta( B) 
&=  \sum_{ C \in \mathcal C_n} \alpha \times \beta ( C) \\
 &=  \sum_{ D \in \mathcal D_n^{(0)}} \alpha \times \beta( D)  +  \sum_{ D \in \mathcal D_n^{(1)}} \alpha \times \beta( D) 
\end{align*}
since every double box of $\mathcal B_k$ is the union of finitely many boxes of $\mathcal C_n$, and since every double box of $\mathcal C_n$ is the union of finitely many boxes of $\mathcal D_n^{(0)}$ and $\mathcal D_n^{(1)}$. 

Similarly,
\begin{align*}
\sum_{k=1}^n \sum_{ B' \in \mathcal B_k'} \alpha \times \beta ( B') 
&=  \sum_{ C' \in \mathcal C_n'} \alpha \times \beta ( C') \\
 &=  \sum_{ D \in \mathcal D_n^{(0)}} \alpha \times \beta ( D)  +  \sum_{ D \in \mathcal D_n^{(2)}} \alpha \times \beta ( D) .
\end{align*}

Taking the difference between these partial sums,
\begin{align*}
\biggl|
\sum_{k=1}^n \sum_{ B \in \mathcal B_k} \alpha \times \beta( B) 
- \sum_{k=1}^n \sum_{ B' \in \mathcal B_k'} & \alpha \times \beta( B') 
\biggr|\\
&\leq   \sum_{ D \in \mathcal D_n^{(1)}} \bigl| \alpha \times \beta( D)\bigr|  +  \sum_{ D \in \mathcal D_n^{(2)}} \bigl| \alpha \times \beta( D) \bigr|\\
&\prec  \frac{R^{(8-d)n}}{r^{8n}}R^{4\nu n} =\left( R^{8-d+4\nu} r^{-8}  \right)^n
\end{align*}
since each $ D =  B_1 \times  B_2 \in \mathcal D_n$ has side lengths $\prec R^n$, so that $\alpha(B_1) \prec L_{m_0}(B_1)^\nu \prec R^{2\nu n}$ and a similar estimate holds for $\beta(B_2)$. 

Letting $n$ go to infinity and using the hypothesis that $\nu > 2\left( \frac{\log r}{\log R} -1  \right) + \frac d4$, we conclude that 
$$
\sum_{k=1}^\infty \sum_{ B \in \mathcal B_k} \alpha \times \beta( B) 
=\sum_{k=1}^\infty \sum_{ B' \in \mathcal B_k'} \alpha \times \beta( B')
$$ 
as required. 
\end{proof}

We are now going to apply this to a specific domain $\Omega$. 

\begin{lem}
\label{lem:FundDomain}
For every complex structure $m_0$ on $S$ and every $d>3$, there exists an open subset $\Omega \subset DG(\widetilde S)$ such that
\begin{enumerate}
\item $\Omega$ is relatively compact in $G(\widetilde S) \times G(\widetilde S)$; 
\item the quotient map $DG(\widetilde S) \to DG(S)=DG(\widetilde S)/\pi_1(S)$ is injective on $\Omega$, and sends $\Omega$ to an open dense set of $DG(S)$;
\item there exists finitely many $\gamma \in \pi_1(S)$ such that $\Omega \cap \gamma(\Omega) \not = \varnothing$;
\item the frontier $\delta \Omega$ of $\Omega$ in $G(\widetilde S) \times G(\widetilde S)$  has Minkowski $m_0$--dimension $<d$.
\end{enumerate}
\end{lem}

Namely, $\Omega$ is a fundamental domain for the action of $\pi_1(S)$ on $DG(\widetilde S)$, whose frontier has small dimension. 

\begin{proof}
The complex structure $m_0$ defines a projection $p \col DG(\widetilde S) \to \widetilde S$, which to a double geodesic $(g,h) \in DG(\widetilde S)$ associates the intersection point $g\cap h \in \widetilde S$. (Note that this projection map depends on $m_0$). 

Let $\omega$ be a compact fundamental polygon for the action of $\pi_1(S)$ on the universal cover $\widetilde S$, bounded by a piecewise differentiable curve. Let $\Omega$ be the preimage of the interior of $\omega$ under $p$. The first three conclusions of the statement clearly hold. 

The frontier $\delta \Omega$ of $\Omega$ in $G(\widetilde S) \times G(\widetilde S)$ is the union of $p^{-1}(\delta \omega)$ and of the subset of the diagonal consisting of those $(h,h)\in G(\widetilde S) \times G(\widetilde S)$ such that $h$ meets $\omega$. Since the boundary of $\omega$ is piecewise differentiable, it follows that the frontier $\delta\Omega$ has Minkowski  $m_0$--dimension $3$, which is less than $d$ by hypothesis. 
\end{proof}

\begin{lem}
\label{lem:IntIndeptDomain}
Given numbers $d>3$ and $0<r\leq \frac12 \leq R<1$, consider two cross-ratios $\alpha $, $\beta \in \mathcal X(S)$  that are $\nu$--H\"older regular with respect to $m_0$ for  $\nu > \max \{  \frac d4  \frac{\log r}{\log R}, 2\left( \frac{\log r}{\log R} -1  \right) + \frac d4\}$. Let $\Omega\subset DG(\widetilde S)$ be provided by Lemma~{\upshape\ref{lem:FundDomain}}, and consider the sum
$$
\mathrm i_{\Omega} (\alpha, \beta) = \sum_{B \in \mathcal B} \alpha \times \beta(B)
$$
as in Lemmas~{\upshape \ref{lem:IntersectionConverges}} and {\upshape \ref{lem:IntIndeptBoxes}}. This number $\mathrm i_{\Omega} (\alpha, \beta)$
is independent of the choice of  $\Omega$.
\end{lem}

\begin{proof}
Let $\Omega'$ be another domain as in Lemma~\ref{lem:FundDomain}.  By Condition~(3) of  Lemma~\ref{lem:FundDomain}, there are finitely many elements $\gamma_1$, \dots, $\gamma_t \in \pi_1(S)$ such that $\Omega$ is covered by the union of the $\gamma_s(\Omega')$. Considering the domains $\Omega_s = \Omega \cap \gamma_s(\Omega')$, we first prove  that
$$
\mathrm i_{\Omega} (\alpha, \beta) = \sum_{s=1}^t \mathrm i_{\Omega_s} (\alpha, \beta) .
$$

For this, consider the subdivision scheme $\{ \mathcal I_n \}_{n \in \N}$, $\{ \mathcal B_n \}_{n \in \N}$ for $\Omega$ provided by Lemma~\ref{lem:GoodBoxes}.  

We then use the same interval family $\{ \mathcal I_n \}_{n \in \N}$ to create a subdivision scheme $\{ \mathcal B_n'' \}_{n \in \N}$ for $\Omega'' = \bigcup_{s=1}^t \Omega_s$. More precisely, inductively define the family $\mathcal B_n''$ to consist of all double boxes $B'' =I_1''\times I_2'' \times I_3''\times I_4''$ with all four $I_i''$ in $\mathcal I_n$, such that:
\begin{enumerate}
\item the double box $B''$ is contained in $\Omega''$;
\item if $I_i \in \mathcal I_{n-1}$ is the level $n-1$ interval that contains $I_i''$, the double box $B =I_1 \times I_2 \times I_3 \times I_4$ is not  contained in $\Omega''$. 
\end{enumerate} 

As in the proof of Lemma~\ref{lem:IntIndeptBoxes}, let $\mathcal C_n$ be the family of all double boxes $B=I_1 \times I_2 \times I_3 \times I_4$, with all four $I_i$ in $\mathcal I_n$, that are contained in $\Omega$, and let $\mathcal C_n'' \subset \mathcal C_n$ be similarly associated to $\Omega''$. Then, by finite additivity,
$$
\sum_{k=1}^n  \sum_{B \in \mathcal B_k}  \alpha \times \beta(B) = \sum_{C \in \mathcal C_n}  \alpha \times \beta(C)
$$
so that 
$$
i_\Omega (\alpha, \beta) = \lim_{n \to \infty}  \sum_{C \in \mathcal C_n}  \alpha \times \beta(C).
$$
Similarly, 
$$
\sum_{s=1}^t \mathrm i_{\Omega_s} (\alpha, \beta) =
i_{\Omega''} (\alpha, \beta) = \lim_{n \to \infty} \sum_{C'' \in \mathcal C_n''}  \alpha \times \beta(C'').
$$

Each double box $C\in \mathcal C_n$ has side lengths of order between $r^n$ and $R^n$, so that its contribution $\alpha \times \beta(C)$ to the above sums is bounded in absolute value by $R^{4\nu n}$. Also, the complement $\mathcal C_n - \mathcal C_n''$ consists of those $C\in \mathcal C_n$ which meet the union of the frontiers $\delta \Omega_s$. Since these frontier have Minkowski $m_0$--dimension $<d$, the now usual volume arguments show that the cardinal of  $\mathcal C_n - \mathcal C_n''$ is bounded by $\frac{R^{(4-d)n}}{r^{4n}}$. Therefore,
\begin{align*}
\biggl |
\sum_{C \in \mathcal C_n}  \alpha \times \beta(C)
- \sum_{C'' \in \mathcal C_n''}  \alpha \times \beta(C'')
\biggr |
&\leq
 \sum_{C \in \mathcal C_n - \mathcal C_n''} 
\bigl |
 \alpha \times \beta(C)
\bigr | \\
&\leq \frac{R^{(4-d)n}}{r^{4n}} R^{4\nu n}.
\end{align*}
Passing to the limit as $n \to \infty$, we conclude that 
$$
i_\Omega (\alpha, \beta) = i_{\Omega''} (\alpha, \beta) = \sum_{s=1}^t \mathrm i_{\Omega_s} (\alpha, \beta) 
$$
since $R^{4-d+4\nu}r^{-4}<1$ by our hypothesis on $\nu$. 

Considering the domains  $\gamma_s^{-1}(\Omega_s)$ in $\Omega'$, the same argument shows that
$$
\mathrm i_{\Omega'} (\alpha, \beta) = \sum_{s=1}^t \mathrm i_{\gamma_s^{-1}(\Omega_s)} (\alpha, \beta) .
$$
In addition, because the cross-ratio functions $\alpha $ and $\beta $ are invariant under the action of the fundamental group,
$$
\mathrm i_{\Omega_s} (\alpha, \beta) = \mathrm i_{\gamma_s^{-1}(\Omega_s)} (\alpha, \beta) 
$$
for every $s$. (Note that $\gamma_s$ distorts the metric induced by the complex structure $m_0$ on the circle at infinity $\partial_\infty \widetilde S$ by a uniformly bounded Lipschitz factor, so that all estimates are preserved.)

It follows that $\mathrm i_{\Omega} (\alpha, \beta) =\mathrm i_{\Omega'} (\alpha, \beta) $. 
\end{proof}

This proves Theorem~\ref{thm:Main1}, in the following form.

\begin{thm}
\label{thm:Main3}
The above construction provides a  well-defined geometric intersection number $i(\alpha, \beta)$ for any two cross-ratio functions $\alpha$, $\beta \in \mathcal X(S)$ that are $\nu$--H\"older regular with respect to the complex structure $m_0 \in \T(S)$ and for some $\nu > \frac 34$. 

\emph{A priori}, this intersection number may depend on the complex structure $m_0$ with respect to which $\alpha$ and  $\beta$ are $\nu$--H\"older regular. However, it is a locally constant function of $m_0 \in \T(S)$. 
\end{thm}

\begin{proof}
Pick numbers $d<4$ and $0<r<\frac 12 <R<1$ sufficiently close to $3$ and $\frac12$, respectively, that
$
\nu > \max \{  \frac d4  \frac{\log r}{\log R}, 2\left( \frac{\log r}{\log R} -1  \right) + \frac d4\}
$. 

Choose a domain $\Omega \subset DG(\widetilde S)$ as in Lemma~\ref{lem:FundDomain}, and consider  the family $\mathcal B = \bigcup_{n=1}^\infty \mathcal B_n$ of double boxes in $\Omega $ provided by Lemma~\ref{lem:GoodBoxes}. Then define
$$
i(\alpha, \beta) = i_\Omega (\alpha, \beta) = \sum_{B \in \mathcal B} \alpha \times \beta(B). 
$$
Lemmas~\ref{lem:IntersectionConverges}, \ref{lem:IntIndeptBoxes} and \ref{lem:IntIndeptDomain} show that this sum converges, and is independent of choices.

In addition, by Lemma~\ref{lem:HolderStructure}, the estimates are stable under small perturbation of the complex structure $m_0$. This guarantees the invariance under small perturbation of the complex structure $m_0$. (The introduction of the numbers $r$, $R$ and $d$ in the various statements  were specially designed to guarantee this.)
\end{proof}

\section{The geometric estimate}
\label{sect:GeomEst}

This section is devoted to the proof of Proposition~\ref{prop:InfLiouvRegular}, which says that infinitesimal Liouville cross-ratios are $\nu$--H\"older regular for every $\nu<1$. It will enable us to apply Theorem~\ref{thm:Main3} to infinitesimal Liouville cross-ratios.

Let $V\in T_{m_0}\mathcal T(S)$ be a tangent vector to the
Teichm\"uller space $\mathcal T(S)$ at the point $m_0$. The
infinitesimal deformation associated with $V$ is described by an
equivalence class $[\mu ]$ of Beltrami coefficients, where the Beltrami coefficients $\mu$ and $\mu'$ are equivalent if and only if
$$
\int_S\mu\varphi =\int_S\mu'\varphi
$$
for all holomorphic quadratic differentials $\varphi$ on $S$. We then define the \emph{Teichm\"uller norm} $\Vert V \Vert$ of the tangent vector $V\in T_{m_0} \mathcal T(S)$ as 
$$
\Vert V \Vert = \min_\mu \Vert \mu \Vert_\infty, 
$$
where the infimum is taken over all Beltrami differentials $\mu$ representing $V$.

\begin{lem}
\label{lem:GeomEst}
There exists a universal constant $c_0>0$ such that
$$
|L_V(I,J)|\leq c_0\,\| V\|\, L_{m_0}(I,J)\, |\log L_{m_0}(I,J)|
$$
for all intervals $I$, $J\subset \partial_{\infty}\tilde{S}$ with
disjoint closure such that $L_{m_0}(I,J)\leq \frac{1}{2}$.
\end{lem}

To simplify the notation, set $L= L_{m_0}(I,J)$. 
\begin{proof} 
 For the complex structure $m_0$, biholomorphically identify the universal cover $\widetilde S$ to  the upper half-plane $\mathbb H^2 \subset \C$. By invariance of the requested estimate under biholomorphic transforms of $\mathbb H^2$, we can further arrange that $I=[- \mathrm e^{\alpha},-1]$ and
$J=[1,\mathrm e^{\alpha}]$ for some $\alpha>0$. An  easy computation then shows that
$L_{m_0}(I,J) = 2 \log \bigl( \cosh \frac\alpha2 \bigr) \asymp \alpha^2$.

The first variation formula for the solution of the Beltrami
equation with coefficient $t\mu$ yields the formula (see
\cite{Wolpert1})
$$
L_V([a,b],[c,d])=-\frac{2}{\pi}Re\int_{\mathbb H^2}\mu
(z)\frac{(a-b)(c-d)}{(z-a)(z-b)(z-c)(z-d)}dx\,dy
$$
where $z=x+iy$.

In our situation, this provides a bound
$$
|L_V(I,J)|\prec  \Vert\mu \Vert_\infty \alpha^2 \int_{\mathbb
H^2}\frac{dx\,dy}{|(z+1)(z+e^\alpha)(z-1)(z-e^{\alpha})|},
$$
 since the intervals $I$ and $J$ have length $\asymp \alpha$.

An integration along the lines of \cite[page 448]{Saric1} then gives
$$
\int_{\mathbb
H^2}\frac{dxdy}{|(z+1)(z+e^{\alpha})(z-1)(z-e^{\alpha})|}\prec  |\log \alpha\,|  \asymp  |\log  L_{m_0}(I,J)\,|
$$

Combining these two inequalities, taking the infimum over all Beltrami coefficients $\mu$ representing $V$, and observing that the constants hidden in the symbols $\prec$ and $\asymp$ are universal  completes the proof. 
\end{proof}

Lemma~\ref{lem:GeomEst} proves the H\"older regularity property of Proposition~\ref{prop:InfLiouvRegular}.


\section{Proof of Theorem~\ref{thm:Main2}}

We restate Theorem~\ref{thm:Main2} for the convenience of the reader. 
\begin{thm}
Let $t\mapsto m_t$ and $u \mapsto n_u$ be two differentiable curves in $\T(S)$, respectively tangent to the vectors $V$ and $W\in T_{m_0}\T(S)$ at $m_0=n_0$. Consider the associated Liouville geodesic currents $L_{m_t}$ and $L_{n_u}$, and the infinitesimal Liouville currents $L_V=\frac d{d t} L_{m_t}{}_{|t=0}$ and $L_W=\frac d{d u} L_{n_u}{}_{|u=0}$. Then,
$$
\frac\partial{\partial t} \frac\partial{\partial u} i(L_{m_t}, L_{n_u})_{|(t,u)=(0,0)} = i(L_V, L_W),
$$
where $i(L_V, L_W)$ is the geometric intersection number provided by Theorem~{\upshape\ref{thm:Main1}} and Proposition~{\upshape\ref{prop:InfLiouvRegular}}, and where $i(L_{m_t}, L_{n_u})$ is the classical intersection number of measure geodesic currents as in \cite{Bonahon1}. 
\end{thm}

\begin{proof}
As in the construction of intersection numbers in the proof of Theorem~\ref{thm:Main2},  pick numbers $d>3$, $0<r<\frac12 <R<1$ and $\nu<1$ such that $
\nu > \max \{  \frac d4  \frac{\log r}{\log R}, 2\left( \frac{\log r}{\log R} -1  \right) + \frac d4\}
$. For every $\alpha$, $\beta\in \mathcal X(S)$ that are $\nu$--H\"older regular with respect to $m_0$, we can then define the intersection number 
$$
i(\alpha, \beta) = \sum_{B \in \mathcal B} \alpha \times \beta(B)
$$
as in the proof of Theorem~\ref{thm:Main2}. 

Lemma~\ref{lem:HolderStructure} and Proposition~\ref{prop:InfLiouvRegular} show that, for $t$, $t'$, $u$, $u'$ sufficiently close to $0$, the cross-ratio functions $L_{m_t}$, $L_{n_u}$, $\frac\partial {\partial t} L_{m_t}$ and $\frac\partial {\partial u} L_{n_u}$ are $\nu$--H\"older regular with respect to the complex structures $m_{t'}$ and $n_{u'}$. Similarly, the family $\mathcal B = \bigcup_{i=1}^\infty \mathcal B_n$ of double boxes satisfy the size estimates of Lemma~\ref{lem:GoodBoxes} with respect to the complex structures $m_{t'}$ and $n_{u'}$.

For a given double box $B = B_1 \times B_2 \in \mathcal B$, with $B_1$, $B_2 \subset G(\widetilde S)$, write
\begin{align*}
\frac\partial {\partial u} \bigl( L_{m_t}\times L_{n_u} (B)\bigr)  &= \frac\partial {\partial u} \bigl( L_{m_t}(B_1) \times L_{n_u} (B_2) \bigr)
=  L_{m_t}(B_1) \times \biggl ( \frac\partial {\partial u} L_{n_u} (B_2) \biggr) \\
&= L_{m_t} \times  \biggl ( \frac\partial {\partial u} L_{n_u}\biggr)  (B) 
\end{align*}
Summing over all boxes $B\in \mathcal B$ and using the fact that all convergence estimates are uniform in $t$, 
$$
\frac\partial {\partial u} i( L_{m_t}, L_{n_u}) = \sum_{B \in \mathcal B} L_{m_t} \times  \biggl( \frac\partial {\partial u} L_{n_u}  \biggr) (B)
= i \biggl(  L_{m_t} ,  \frac\partial {\partial u} L_{n_u} \biggl).
$$

Iterating this argument (and again using the uniform convergence estimates) then gives
$$
\frac\partial {\partial t}\frac\partial {\partial u} i( L_{m_t}, L_{n_u}) = \sum_{B \in \mathcal B}  \biggl( \frac\partial {\partial t} L_{m_t} \biggr) \times  \biggl( \frac\partial {\partial u} L_{n_u}  \biggr) (B)
= i \biggl(  \frac\partial {\partial t}  L_{m_t} ,  \frac\partial {\partial u} L_{n_u} \biggl)
$$
for every $(t,u)$ sufficiently close to $(0,0)$, and in particular for $(t,u)=(0,0)$.
\end{proof}

\end{document}